\documentclass[pdflatex,sn-mathphys-num]{sn-jnl}
\pdfoutput = 1

\usepackage{graphicx}%
\usepackage{multirow}%
\usepackage{amsmath,amssymb,amsfonts}%
\usepackage{amsthm}%
\usepackage{mathrsfs}%
\usepackage[title]{appendix}%
\usepackage{xcolor}%
\usepackage{textcomp}%
\usepackage{manyfoot}%
\usepackage{booktabs}%
\usepackage{algorithm}%
\usepackage{algorithmicx}%
\usepackage{algpseudocode}%
\usepackage{listings}%
\usepackage{tikz}
\usetikzlibrary{positioning}

\newtheorem{thm}{Theorem}[section]
\newtheorem{prop}[thm]{Proposition}
\newtheorem{lem}[thm]{Lemma}

\newtheorem{conj}[thm]{Conjecture} 
\newtheorem{rmk}[thm]{Remark}

\newcommand{\vol}{\mathrm{Vol}}

\begin{document}

\title[Mirzakhani's frequencies of simple closed geodesics in large genus]{Mirzakhani's frequencies of simple closed geodesics on hyperbolic surfaces in large genus and with many cusps}
\author{Irene Ren}

\abstract{
We present a proof of a conjecture proposed by V. Delecroix, E.~Goujard, P. Zograf, and A. Zorich, which describes the large genus asymptotic behaviours of the ratio of frequencies of separating over nonseparating simple closed geodesics on a closed hyperbolic surface of genus $g$ with $n$ cusps. We explicitly give the function $f(\frac{n}{g})$ in the conjecture. The moderate behaviour of the frequencies with respect to the growth rate of the number of cusps compared to that of the genus drastically contrasts with the behaviour of other geometric quantities and exhibits the topological nature of the frequencies.
}

\maketitle

\tableofcontents

\section{Introduction}
\label{sec:intro}

    Let $X$ be a hyperbolic surface of genus $g$ with $n$ cusps. A closed geodesic is called simple if it does not have self-intersections. A simple closed geodesic is called separating if it splits the hyperbolic surface $X$ into two parts, and is called nonseparating if it does not.
    
    In 2008 in her paper \cite{Mirzakhani:2008ola}, M. Mirzakhani has proved that for any multicurve $\gamma$ and any hyperbolic surface $X$ in $\mathcal{M}_{g,n}$ the number $s_X (L,\gamma)$ of simple closed geodesics multicurves on $X$ of topological type $\gamma$ and of hyperbolic length at most $L$ has asymptotics:
    $$s_X (L,\gamma ) \sim \mu_{Th}(B_X) \frac{c(\gamma)}{b_{g,n}}\cdot L^{6g-6+2n}, \quad L \to \infty.$$
    Here the Thurston measure $\mu_{Th}(B_X)$ depends only on the hyperbolic metric of $X$, $b_{g,n}$ only depends on the topological properties of $X$. The frequency of geodesics with topology type $\gamma$, denoted $c(\gamma)$, only depends on the topology type of $\gamma$, and admits a formula in terms of the intersection numbers of $\psi$-classes.

    The frequency of separating simple closed geodesics, $c_{g,n,\mathrm{sep}}= \sum_{\gamma} c(\gamma)$, is defined as a summation over all topological types of separating simple closed curve $\gamma$. In nonseparating case, $c_{g,n,\mathrm{nonsep}} = c(\gamma)$, where $\gamma$ is the topological type of nonseparating simple closed curve. 

    M. Mirzakhani in \cite{Mirzakhani:2008ola} has explored the frequencies of different types of simple closed geodesics on hyperbolic surfaces of genus $g$ with $n$ cusps, and their relationship with the Weil-Petersson volume of moduli spaces of corresponding bordered Riemann surfaces. She proved that the ratio of frequencies, $\frac{c_{g,n,\mathrm{sep}}}{c_{g,n,\mathrm{nonsep}}}$ is a topological quantity, which is one and the same for all hyperbolic surfaces of genus $g$ and $n$ cusps, no matter the geometric structure of the surface.

    In \cite{DGZZ21}, V. Delecroix, E. Goujard, P. Zograf and A. Zorich related the count of asymptotic frequencies of simple closed geodesics on a hyperbolic surface with that of square-tiled surfaces through Witten–Kontsevich correlators. In a more recent paper \cite{dgzz23}, the same authors introduced a correspondence between squared-tiled surfaces and meanders. In the same paper, they discussed the frequencies of separating versus nonseparating simple closed geodesics on hyperbolic surfaces of a large genus $g$ with $n$ cusps. 
    
	In this paper we are dealing with a conjecture from \cite{dgzz23}, which describes the large $g$ asymptotic behaviour of the ratio of $c_{g,n,\mathrm{sep}}$ and $c_{g,n,\mathrm{nonsep}}$ on hyperbolic surfaces. This conjecture indicates the topological nature of the asymptotic frequencies of simple closed geodesics on hyperbolic surfaces of a large genus $g$ with $n$ cusps.

\subsection*{Acknowledgement}
I thank Anton Zorich for introducing this problem to me as well as reviewing the scripts. I thank Simon Barazer and Anton Zorich for useful discussions.

\section{Statement of the Conjecture}

\subsection{The Conjecture}

V. Delecroix, E. Goujard, P. Zograf and A. Zorich has proved in \cite{DGZZ21} that $\frac{c_{g,0,sep}}{c_{g,0,nonsep}} \sim \sqrt{\frac{2}{3\pi g}} \cdot \frac{1}{4^g}$, as $g \to \infty$. Later they accordingly conjectured in \cite{dgzz23} that for larger number of cusps, the ratio should be: \newline

\begin{conj} \label{thm:conjMain}
(Conjecture 2.16 in \cite{dgzz23})
The ratio of frequencies of separating over nonseparating simple closed geodesics on a closed hyperbolic surface of genus g with n cusps admits the following uniform asymptotics:
\begin{equation}
    \frac{c_{g,n,\mathrm{sep}}}{c_{g,n,\mathrm{nonsep}} }=\sqrt{\frac{2}{3\pi g}} \cdot \frac{1}{4^g} \cdot f\left(\frac{n}{g}\right) \cdot \left(1+\varepsilon(g,n) \right),
\end{equation}
where the function $f:[0,\infty) \mapsto \mathbb{R}$ is continuous and increases monotonously from $f(0) = 1$ to $f(\infty) = \sqrt{2}$ and the error term $\varepsilon(g,n)$ tends to $0$ as $g \rightarrow \infty$ uniformly in $n$.
\end{conj}

M. Mirzakhani has proved in \cite{Mirzakhani:2008ola} that the ratio of frequencies of a given hyperbolic surface depends only on the topological properties, i.e. the genus $g$ and the number of cusps $n$. An explicit example is given by Bers' hairy torus, which is a torus of $n^2$ cusps. We can take a very symmetric hairy torus such as in \cite{Buser1992GeometryAS} section 5.3, or take a random shape hairy torus with the same number of cusps, but the ratio $\frac{c_{g,n,\mathrm{sep}}}{c_{g,n,\mathrm{nonsep}} }$ will always be $\frac{1}{6}$.

 The large $g$ asymptotic of $\frac{c_{g,n,\mathrm{sep}}}{c_{g,n,\mathrm{nonsep}} }$ when $1 \ll g \ll n$ (Remark 2.14 in \cite{dgzz23}) and $1 \ll n \ll g$ (Theorem 2.15 in \cite{dgzz23}) are already proved in \cite{dgzz23}, which reveals that the function $f$ has $f(0)=1$ and $f(\infty) = \sqrt{2}$.

The above conjecture clarifies that the ratio has a very controlled behaviour under the variation of $\frac{n}{g}$ when $g$ is large, in contrast to geometric quantities such as the spectral gap of Laplacian operators and Cheeger constant. The discussion of Witten–Kontsevich correlators includes works of A. Aggarwal, in \cite{aggarwal2021large}, where it was proved that Witten–Kontsevich correlators, after normalization, are uniformly close to 1 in the regime $n^2 \ll g$ and might explode exponentially otherwise. 

M. Mirzakhani investigated the geometric properties of random hyperbolic surfaces of large genus without cusps in \cite{Mirzakhani:2010pla}, and proved that there exists a spectral gap of 0.00247. This pioneering work of Mirzakhani was extended to the result of Y. Wu and Y. Xue \cite{Wu_2022} and that of M. Lipnowski and A. Wright \cite{lipnowski2023optimal}, showing separately an estimation as $\frac{3}{16} $ of spectral gaps of random hyperbolic surfaces. Furthermore, N. Anantharaman and L. Monk give a more accurate estimate of $\frac{2}{9}$ for this spectral gap in \cite{anantharaman2023friedmanramanujan}. Mirzakhani's work was also extended to the case of random hyperbolic surfaces of large genus with cusps. W. Hide in \cite{Hide2022} studied the spectrum of the Laplace operator on hyperbolic surfaces when $n^2 \ll g$, and proved the existence of the positive spectral gap; while Y. Shen and Y. Wu studied in \cite{shen2023arbitrarily} the asymptotic behavior of the Cheeger constants and spectral gaps of random hyperbolic surfaces when $n^2 \gg g$, and by their results, there is no uniformly positive spectral gap for Weil–Petersson random hyperbolic surfaces in that regime.

The contrasting behaviour of the frequencies when $\frac{n}{g}$ changes, in comparison to geometric quantities like spectral gap and Cheeger constant, emphasizes the topological nature of the frequencies.

\subsection{Asymptotic Form}

 Consider a multicurve $\gamma$ defined as $\sum_{i=1}^k H_i \gamma_i$, where $\gamma_i$'s are pairwise disjoint simple closed curves which are non-trivial, non-peripheral; and the weights $H_i$'s are integers assigned to each $\gamma_i$. Two multicurves are of the same topological type if they are in the same mapping class group orbit. 
    
    The dual graph $\Gamma$ to a multicurve $\gamma$ is called a stable graph, in which each vertex of the stable graph is associated to the genus of the corresponding connected component, and the half edge at the corresponding vertex represents each puncture on that component. Stable graph encode the topological type of the multicurve. Usually, when $\gamma$ is a nonseparating multicurve on a closed hyperbolic surface of genus $g$ with $n$ cusps, we denote the corresponding stable graph $\Gamma_{g,n}$; when $\gamma$ is separating, we denote the corresponding stable graph $\Gamma_{g_1,n_1}^{g_2,n_2}$, where $g_1,g_2$ and $n_1, n_2$ are the genus and number of cusps of the components of the original hyperbolics surface. Figure \ref{fig:stable_graph} demonstrates the examples of separating and nonseparating stable graphs.   

    \begin{figure}
        \centering
        \includegraphics[width=0.7\linewidth]{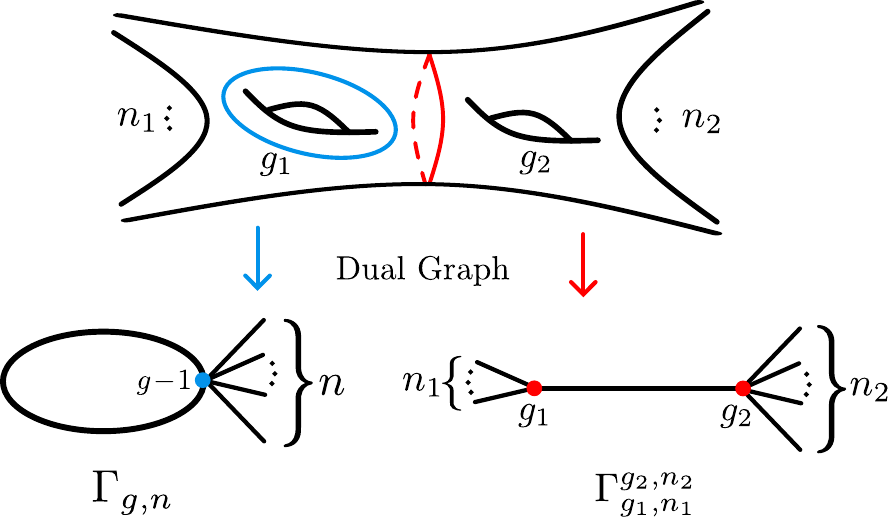}
        \caption{Non-separating (blue) and separating (red) simply closed curves for hyperbolic surfaces of genus $g$ with $n$ cusps (up), and their corresponding stable graphs (bottom left/right)}
        \label{fig:stable_graph}
    \end{figure}

    A square-tiled surface admits a decomposition into maximal horizontal cylinders filled with isometric closed regular flat geodesics. Using a similar correspondence one can show that stable graphs encode also the type of decomposition into cylinders \cite{dgzz23}. For any topological class $\gamma$ of simple closed multicurves, the associated Mirzakhani’s asymptotic frequency of hyperbolic multicurves coincides with the asymptotic frequency of simple closed flat geodesic multicurves of same topological type represented by associated square-tiled surfaces. 

    The moduli space of quadratic differentials is the total space of cotangent bundles of $\mathcal{M}_{g,n}$, therefore it naturally contains a canonical symplectic structure and a volume element. Square-tiled surfaces represent integer points in this moduli space, and asymptotic of the number of square-tiled surfaces of genus $g$ with $n$ conical points tiled with at most $N \to +\infty$ squares gives us the volume of the moduli space, which is called the Masur–Veech volume.

    We use $cyl_1(\Gamma)$ to denote the contribution to the Masur-Veech volume of $\mathcal{Q}_{g,n}$, the moduli space of meromorphic quadratic differentials, that comes from single-band square-tiled surfaces corresponding to the stable graph $\Gamma$. In that case sometimes we also use  $\Gamma_1$ instead of $\Gamma$, to emphasize that the surface is single band square-tiled. The contribution to that of stable graphs corresponding to a surface of genus $g$ with $n$ punctures is called $\text{Vol}(\Gamma_{g,n})$. Likely, a volume contribution of multicurves corresponding to a surface of genus $g$ with $n$ punctures is called $\text{Vol}(\gamma)$. Sometimes we also use the notation $\text{Vol}(\Gamma, \textbf{H})$ for the volume element for a stable graph of a certain given weight $\textbf{H}= (H_1, \dots, H_k)$.
    
     In \cite{dgzz23}, Proposition 6.4 explicitly provided the contribution to Masur-Veech volumes of principal stratum $\mathcal{Q}_{g,n}$ of meromorphic quadratic differentials, which comes from single-band square-tiled surfaces corresponding to all stable graphs. While in \cite{DGZZ21}, Theorem 1.22 gave a relation between Mirzakhani’s asymptotic frequency of closed geodesic multicurves of certain topological type and the volume contribution of the corresponding stable graph to Masur-Veech volume. Accordingly one has explicitly the ratio of frequencies of separating and nonseparating geodesics. \newline

    \begin{thm}
        (Theorem 1.22 in \cite{DGZZ21}) Let $(g,n)$ be a pair of nonnegative integers satisfying $2g+n>3$ and different from $(2,0)$. Let $\gamma \in \mathcal{ML}_{g,n}(\mathbb{Z})$ be a multicurve, and let $(\Gamma, \mathbf{H})$ be the associated stable graph and weights. Then the volume contribution $\mathrm{Vol} (\Gamma, \mathbf{H})$ to the Masur–Veech volume $\mathrm{Vol}\mathcal{Q}_{g,n}$ coincides with Mirzakhani’s asymptotic frequency $c(\gamma)$ of closed geodesic multicurves of topological type $\gamma$ up to an explicit factor depending only on $g$ and $n$:
    \begin{equation}
        \mathrm{Vol} (\Gamma, \mathbf{H}) = 2(6g-6+2n) \cdot (4g-4+n)! \cdot 2^{4g-3+n} \cdot c(\gamma).
    \end{equation}
    \end{thm}
    The above theorem allows us to write $\frac{c_{g,n,\mathrm{sep}}}{c_{g,n,\mathrm{nonsep}} }$ as a ratio of volume contributions of stable graphs, which we give here:
    \begin{equation}
        \frac{c_{g,n,\mathrm{sep}}}{c_{g,n,\mathrm{nonsep}}}= \frac{\frac{1}{2} \sum_{n_1=0}^{n} \binom{n}{n_1}\sum_{g_1=0}^{g} |\mathrm{Aut}\Gamma^{g_2,n_2}_{g_1,n_1}| |\mathrm{Vol}\left( \Gamma^{g_2,n_2}_{g_1,n_1} \right) | }{\text{Vol}(\Gamma_{g,n})}.
    \end{equation}
    
    The following formula is true for any stable graph $\Gamma$ with a single edge, and for any $g$ and $n$, as maintained by a generalization of Theorem 3.1 in \cite{dgzz23}:
    
    \begin{equation*}
        \mathrm{Vol}(\Gamma)=cyl_1{\Gamma} \cdot \zeta (6g-6+2n).
    \end{equation*}
      
      In our case when considering large genus asymptotics, one always has $\mathrm{Vol}(\Gamma) \sim cyl_1{\Gamma} $, since $\zeta (6g-6+2n)$ tends to $1$ exponentially fast as $g \rightarrow \infty$.
    
    Furthermore, $cyl_1{\Gamma}$ has an explicit expression as a sum of products of binomial coefficients and Witten-Kontsevich correlators when $g \geq 2$. On the Deligne-Mumford compactification $\overline{\mathcal{M}}_{g,n}$ of moduli space $\mathcal{M}_{g,n}$, let $\mathcal{L}_{n \geq i \geq 1}$ be the tautological line bundles over $\overline{\mathcal{M}}_{g,n}$, whose fiber at a point $(C, x_1, \dots, x_n)$ is the cotangent space of the curve. The first Chern class of this bundle $c_1(\mathcal{L}_i)$ are referred as the $\psi_i$ class of $\overline{\mathcal{M}}_{g,n}$. Then the intersection of the $\psi$ classes 
    $$\langle \tau_{d_1} \dots \tau_{d_n} \rangle = \int_{\overline{\mathcal{M}}_{g,n}} \psi_1^{d_1} \dots \psi_n^{d_n}$$
    is called the Witten-Kontsevich correlator, where $\sum_{i=1}^n d_i = 3g-3+n$.
    \\
    \begin{prop}
        (Proposition 6.4 in \cite{dgzz23}) Assume $g \geq 2$. The contribution to the Masur-Veech volume of the principle stratum $\mathcal{Q}_{g,n}$ of meromorphic quadratic differentials coming from single-band square-tiled surfaces corresponding to the stable graph $\Gamma_{g,n}$ has the following form:
        \begin{equation}
            cyl_1(\Gamma_{g,n}) = 2^{g+1} \binom{4g-4+n}{g} \cdot g! \sum_{k=0}^{3g-4} \binom{3g-4+2n}{n+k}\langle \tau_k\tau_{g-4-k}\rangle_{g-1}.
        \end{equation}

        The total contribution to the Masur–Veech volume of the principal stratum $\mathcal{Q}_{g,n}$ of meromorphic quadratic diﬀerentials coming from single-band square-tiled surfaces corresponding to all stable graphs $\Gamma^{g_2,n_2}_{g_1,n_1}$  has the following form:

     \begin{equation}
     \label{numerator}
        \frac{1}{2}\sum_{n_1=0}^{n} \binom{n}{n_1} \sum_{g_1=0}^{g}|\mathrm{Aut}\Gamma^{g_2,n_2}_{g_1,n_1}| |\mathrm{Vol}  \left( \Gamma^{g_2,n_2}_{g_1,n_1} \right) |=\frac{2^{g+1}}{24^g} \cdot \binom{4g-4+n}{g} \displaystyle \sum\limits_{g_1=0}^{g} \binom{g}{g_1} \binom{3g-4+2n}{3g_1-2+n}
    \end{equation}
    \end{prop} 
    
    Therefore numerator in $\frac{c_{g,n,\mathrm{sep}}}{c_{g,n,\mathrm{nonsep}} }$ is given directly in \eqref{numerator}.

    Meanwhile, when $g$ is large, the 2-correlators have the following asymptotics:\newline
   \begin{lem}
       The large genus asymptotics for 2-correlators is
    \begin{equation}
        \langle\tau_k\tau_{3g-1-k} \rangle_g= \frac{1}{24^g \cdot g!} \cdot \frac{(6g-1)!!}{(2k+1)!!(6g-1-2k)!!} \left(1+O\Big(\frac{1}{g}\Big)\right),
    \end{equation}
       and the error term $O(\frac{1}{g})$ is uniform when $0 \leq k \leq 3g-1$.
   \end{lem}
   The proof of this lemma can be found in \cite{DGZZ21} Proposition 4.1 and Formula~4.2, or with more generality, from \cite{aggarwal2021large}, theorem~1.5.

   Hence the denominator of $\frac{c_{g,n,\mathrm{sep}}}{c_{g,n,\mathrm{nonsep}} }$ has the following asymptotic expression in terms of binomial coefficients:
    \begin{equation}
    \label{eq:denomaymp}
        \vol(\Gamma_{g,n})= \frac{2^{g+3} g}{24^{g} (g-1)} \cdot \binom{4g-4+n}{g} \cdot \sum_{k=0}^{3g-4} \frac{\binom{3g-4+2n}{n+k} \binom{6g-6}{2k+1} }{\binom{3g-4}{k}} \cdot \left(1+O\Big(\frac{1}{g}\Big) \right)\, .
    \end{equation}
    Note that the error term in \eqref{eq:denomaymp} does not depend on $n$.
    
    Consequently, the ratio of frequencies of separating and nonseparating curves has the following asymptotics
    \begin{equation}
    \label{eq:ratio}
        \dfrac{c_{g,n,\mathrm{sep}}}{c_{g,n,\mathrm{nonsep}}} = \frac{1}{4}\ \dfrac{\displaystyle \sum\limits_{g_1=0}^{g} \binom{g}{g_1} \binom{3g-4+2n}{3g_1-2+n}}{\displaystyle \sum\limits_{k=0}^{3g-4} \frac{\binom{3g-4+2n}{n+k} \binom{6g-6}{2k+1} }{\binom{3g-4}{k}}} \cdot \left(1+O\Big(\frac{1}{g}\Big) \right)\, ,
    \end{equation}
    where the error term comes from the asymptotics of the 2-correlators.

\section{Main Theorem and Its Proof}    

\subsection{The Main Theorem}

To prove conjecture \ref{thm:conjMain}, we specify the conjecture by giving explicitly the expression of $f(\frac{n}{g})$, and we restate the conjecture as the theorem below.\newline       
\begin{thm} \label{thm:MainResult}
        Conjecture \ref{thm:conjMain} holds, and
        the unknown function $f$ is given by
        \begin{equation}
         f(\lambda)  =\sqrt{\frac{6+2\lambda}{6+\lambda} }\, ,
        \end{equation}
        where we define $n = \lambda g$, $\lambda \in \mathbb{R}_+$.
\end{thm}
        
\subsection{Large $g,n$ Estimate of the Ratio}
    To prove theorem \ref{thm:MainResult}, we will use the following lemmas.\newline
    
\begin{lem} \label{lem-Estimate}
    (Lemma B.6. in \cite{dgzz23}): 
    \begin{equation}
        \binom{y}{p y} = e^{y H(p)} \frac{1}{\sqrt{2\pi y\, p(1-p)}} \left( 1+ O\left(\frac{1}{y}\right) \right)\, , \quad \text{where} \quad H(p) = - p \log p - (1-p) \log (1-p) \, ,
    \end{equation}
uniformly in $p$ restricted to compact subsets of $(0, 1)$. Here, the binomial coefficient with real parameters is interpreted in terms of $\Gamma$ functions,    
    \begin{equation}
        \binom{a}{b} = \frac{\Gamma(a+1)}{\Gamma(b+1)\Gamma(a-b+1)}\, .
    \end{equation}
\end{lem} 
\begin{proof}
  The proof involves directly applying Stirling’s formula.  
\end{proof}

\begin{lem}
    The function $H(p)$ has the following properties,
    \begin{itemize}
        \item $0 \leq H(p) \leq \log 2$, $\forall p \in (0,1)$.
        \item $H(p)$ obtains its maximal value $\log 2$ at $p=1/2$.
        \item $H(p)$ satisfies the inequality
    \begin{equation} \label{eqn-HinEquality}
        H\left(\frac{1}{2}+\frac{x}{2} \right) \leq \log 2 -\frac{x^2}{2}\, , \qquad \forall  x \in (-1,1)\, .
    \end{equation}
    \end{itemize}
\end{lem} 
\begin{proof}
$H'(p) = \log (1-p) - \log p$ is positive when $0<p<1/2$, negative when $1/2<p<1$, zero when $p=1/2$. Therefore $H(p)$ obtains the maximum at $p=1/2$. The proof of the inequality is similar, as one can verify from the first derivative that $H(1/2+x/2)-\log 2 + x^2/2$ obtains its maximum at $x=0$. 
\end{proof}
\begin{lem} \label{lem:tailEstimate}
(Problem 9.42 in \cite{graham1994concrete})
The contribution of tails (neighbourbood of the endpoints) to the total sum of binomial coefficients admits the following bound:
    \begin{equation}
        \sum_{k =0}^{\lfloor sy \rfloor} \binom{y}{k} < \frac{1-s}{1-2s} \frac{1}{\sqrt{2\pi y s(1-s)}} e^{y H(s)} \left( 1+ O\left(\frac{1}{y}\right)\right) \, , 
    \end{equation}
    when $s \in (0,\frac{1}{2})$. The contribution of the region $k \in [(1-s)y,y]$ is bounded in the same way.
\end{lem}
\begin{proof}
    Using the standard identities of $\Gamma$-function we notice that
    \begin{equation}
        \dfrac{\binom{y}{k-1}}{\binom{y}{k}}= \frac{k}{y-k+1} \leq \frac{s y}{y-sy+1} < \frac{s}{1-s} <1 \, . 
    \end{equation}
    Therefore, the sum is bounded by a geometric sum
    \begin{equation}
        \sum_{k=0}^{\lfloor sy \rfloor} \binom{y}{k} < \binom{y}{sy}  \Big(1 + \frac{s}{1-s} + (\frac{s}{1-s})^2 + \cdots \Big) =  \binom{y}{sy} \frac{1-s}{1-2s} \, . 
    \end{equation}
    When $y \rightarrow \infty$, we use Lemma \ref{lem-Estimate} to expand the binomial coefficient, which completes the proof of the lemma.

The region $k \in [(1-s)y,y]$ can be obtained by a change of variable $k \rightarrow y-k$ so the bound also holds.
\end{proof}

\begin{rmk} \label{rmk:tailEstimate}
    The method of proving Lemma \ref{lem:tailEstimate} can be generalized to products and/or ratios of binomial coefficients. It follows from Lemma \ref{lem-Estimate} that the dominant contributions of such products and/or ratios are of order $2^{y}/\sqrt{y}$. The tail contributions, on the contrary, are of order $(\exp H(s))^y/\sqrt{y}$. For $0<s<1$ the latter is exponentially smaller than the former and can be dropped from the summation.
\end{rmk}

Now we start to estimate the asymptotics of the numerators and denominators in equation \eqref{eq:ratio}. The results are presented in the following two propositions. \newline

    \begin{prop} \label{prop:numAsy}
        Let $\lambda = \frac{n}{g}$.  The asymptotics of the numerator of \eqref{eq:ratio} is given by
        \begin{equation} \label{eqn:numeratorSum}
            \sum\limits_{g_1=0}^{g} \binom{g}{g_1} \binom{3g-4+2n}{3g_1-2+n} = \frac{1}{\sqrt{\pi g (\lambda+6)}}2^{(2\lambda+4)g-4} \big(1 + o(1) \big)\, ,
        \end{equation}
        with the error term uniformly small in $n$ as $g \rightarrow \infty$.
    \end{prop}
    \begin{proof}
    For any $n$ both binomial coefficients obtain their maximal value at $g_1 = g/2$. In the large $g$ limit, the dominant contribution to the sum comes from the region where $g_1$ gets close to $g/2$. Make a change of variable
    \begin{equation*}
        \frac{g_1}{g} = \frac{1}{2}\Big(1 + \frac{x}{\sqrt{g}} \Big)\, .
    \end{equation*}
    As $g_1\in [0,g]$, $x \in [-\sqrt{g},\sqrt{g}]$.
    
    To simplify the expression, we first factor out some finite factors of the second binomial coefficient in \eqref{eqn:numeratorSum} as 
    \begin{equation}
    \label{eq:shift}
            \binom{3g-4+2n}{3g_1-2+n} = \binom{(2\lambda+3)g - 4}{\lambda g +3g_1-2} = \binom{(2\lambda+3)g}{\lambda g +3g_1} \cdot  r_0(g_1;g, \lambda) \, ,
    \end{equation}
        where
    \begin{equation} \label{eqn-num-r0}
        r_0 (g_1;g, \lambda) = \frac{\lambda g +3g_1 }{ (2\lambda+3)g } \cdot \frac{\lambda g +3g_1-1 }{(2\lambda+3)g-1 } \cdot \frac{ \lambda g +3(g-g_1)-1 }{ (2\lambda+3)g-2 } \cdot \frac{ \lambda g +3(g-g_1) }{(2\lambda+3)g-3 }.
    \end{equation}
The numerator of $r_0$ obtains its maximal at $g_1 = g/2$, because it is made of two quadratic polynomials 
\begin{equation*}
    \Big(\lambda g +3g_1\Big) \Big(\lambda g +3(g-g_1)\Big) \times \Big(\lambda g +3g_1-1\Big) \Big(\lambda g +3(g-g_1)-1\Big)
\end{equation*}
that both obtain the maximal value at that point. Therefore, we conclude that $r_0$ is bounded by the value at $g_1 = g/2$:
\begin{equation}
r_0 \leq r_0(\frac{g}{2};g,\lambda) = \frac{1}{32} \left(\frac{3}{g (2 \lambda +3)-3}+\frac{1}{2 g \lambda +3 g-1}+2\right) \, .
\end{equation}
Clearly the first two terms vanishes uniformly in the large $g$ limit, so we conclude
\begin{equation*}
    r_0 \leq \frac{1}{16} \Big( 1 + o(1) \Big) \, .
\end{equation*}

Our strategy is to divide the whole domain into the \textit{tail} region, where $x\in [-\sqrt{g},-(1-\varepsilon) \sqrt{g}] \cup [(1-\varepsilon) \sqrt{g},\sqrt{g}]$ and the rest, where $x \in [-(1-\varepsilon) \sqrt{g},(1-\varepsilon) \sqrt{g}]$. Recall that the variable $x$ is defined in the beginning of the proof as $g_1/g=1/2 (1+x/g) $. The constant $\varepsilon$ satisfies $0<\varepsilon <1$. The Lemma \ref{lem:tailEstimate} and the Remark \ref{rmk:tailEstimate} imply that the contribution of the tail region is negligible. 

We now focus on the region when $x \in [-(1-\varepsilon)\sqrt{g},(1-\varepsilon)\sqrt{g}]$, for any $0<\varepsilon<1$. For the binomial coefficient product, using Lemma \ref{lem-Estimate}, we find
    \begin{equation} \label{eqn-num-binomialAsy}
        \binom{g}{g_1} \binom{3g+2n}{3g_1+n} \sim \frac{1}{2\pi} R(g_1;g,n) \exp\left(g H\left(\frac{g_1}{g}\right) + (3g+2n) H\left(\frac{3g_1+n}{3g+2n}\right)\right)\, ,
    \end{equation}
    where
    \begin{equation}
        R(g_1;g,n) = \dfrac{1}{\sqrt{g(3g+2n)}} \dfrac{1}{\sqrt{\Big(\dfrac{g_1}{g} \Big)\Big(1-\dfrac{g_1}{g} \Big)}} \dfrac{1}{\sqrt{\Big(\dfrac{3g_1+n}{3g+2n}\Big) \Big(1-\dfrac{3g_1+n}{3g+2n} \Big)}} \, .
\end{equation}

Note that
\begin{equation*}
    \frac{3g_1+n}{3g+2n}=\frac{\lambda g+3g_1}{(2\lambda+3)g} = \frac{1}{2} \Big( 1 +  \frac{3 x}{(3+2\lambda)\sqrt{g}} \Big).
\end{equation*}
Thus the remainder term $R(g_1;g,n)$ reads
    \begin{equation}
        R(g_1;g,n) = \frac{4}{g \sqrt{2 \lambda +3}} \frac{1}{\sqrt{1-\frac{x^2}{g}}} \frac{1}{\sqrt{1-\frac{9x^2}{(3+2\lambda)^2 g}}}    
    \end{equation}
Together with the estimate of the product term $r_0$ in \eqref{eq:shift}, we see that the original product of binomials is bounded by
\begin{multline}\label{eqn-Num-EstimatePre}
    \binom{g}{g_1} \binom{3g-4+2n}{3g_1-2+n} \leq \frac{1}{2\pi}\frac{1}{4g \sqrt{2 \lambda +3}} \frac{1}{\sqrt{1-\frac{x^2}{g}}} \frac{1}{\sqrt{1-\frac{9x^2}{(3+2\lambda)^2 g}}} \\
    \times \exp\left[g H\Big(\frac{1}{2} + \frac{x}{2\sqrt{g}}\Big) + (3g+2n) H\Big(\frac{1}{2} + \frac{1}{2} \frac{3 x}{(3+2\lambda)\sqrt{g}} \Big)\right] \Big(1 +o(1) \Big)\, .
\end{multline}
Now we estimate the exponential term. Using the inequality \eqref{eqn-HinEquality}, we find
    \begin{equation} \label{eqn-Num-expUpper}
    \begin{aligned}  
        &g H\Big(\frac{1}{2} + \frac{x}{2\sqrt{g}}\Big) + (3g+2n) H\Big(\frac{1}{2} + \frac{1}{2} \frac{3 x}{(3+2\lambda)\sqrt{g}} \Big) \\
        \leq &\ g\Big(\log 2 - \frac{x^2}{2g} \Big) +(3g+2n) \left(\log 2 - \frac{1}{2}\Big( \frac{3 x}{(3+2\lambda)\sqrt{g}} \Big)^2 \right) \\
        = &\ g (2+\lambda)  \log 4-\frac{(6+\lambda) x^2}{3 +2 \lambda}\, .
    \end{aligned}   
    \end{equation}
Therefore, we have proven
\begin{equation} 
            \binom{g}{g_1} \binom{3g-4+2n}{3g_1-2+n} \leq \frac{1}{\sqrt{1-\frac{x^2}{g}}} \frac{1}{\sqrt{1-\frac{9x^2}{(3+2\lambda)^2 g}}} \frac{2^{(2\lambda +4)g-4}}{\sqrt{(2\lambda + 3)}  \frac{\pi g}{2}} e^{-(1+\frac{9}{2\lambda +3}) \frac{x^2}{2}} \, ,
        \end{equation}
when $x \in [-(1-\varepsilon)\sqrt{g},(1-\varepsilon)\sqrt{g}]$. In that regime the first two $x$-dependent prefactors increase when $x^2$ increases, thus are bounded by the value at $x^2 = (1-\varepsilon)^2 g$. Therefore,
\begin{equation}\label{eqn-Num-UpperBound}
    \begin{aligned}
        & \binom{g}{g_1} \binom{3g-4+2n}{3g_1-2+n} \leq \frac{1}{\sqrt{1-\frac{x^2}{g}}} \frac{1}{\sqrt{1-\frac{9x^2}{(3+2\lambda)^2 g}}} \frac{2^{(2\lambda +4)g-4}}{\sqrt{(2\lambda + 3)}  \frac{\pi g}{2}} e^{-(1+\frac{9}{2\lambda +3}) \frac{x^2}{2}}\, \\
        \leq &\ \frac{1}{\sqrt{1-(1-\varepsilon)^2}} \frac{1}{\sqrt{1-(1-\varepsilon)^2\frac{9}{(3+2\lambda)^2 }}} \frac{2^{(2\lambda +4)g-4}}{\sqrt{(2\lambda + 3)}  \frac{\pi g}{2}} e^{-(1+\frac{9}{2\lambda +3}) \frac{x^2}{2}}\, \\
        \leq &\ \frac{1}{1-(1-\varepsilon)^2} \frac{2^{(2\lambda +4)g-4} }{\sqrt{3}  \frac{\pi g}{2}} e^{- \frac{x^2}{2}}\, ,\\
    \end{aligned}
\end{equation}
where in the last step we substitute the upper bound for each $\lambda$-dependent term to get an uniform upper bound.

On the other hand, when $x \in [-g^{\frac{1}{4}-\delta}, \ g^{\frac{1}{4}-\delta}]$, ($0<\delta<1/4$) we can find a finer estimate. Recall that the $r_0$ term \eqref{eq:shift} is bounded by $\frac{1}{16} \big(1+o(1)\big)$, in this regime, we can prove it saturates this upper bound. We rewrite $r_0$ as
\begin{multline}
r_0 = 
\frac{1}{16} \frac{[(2\lambda+3)g]^4}{[(2\lambda+3)g][(2\lambda+3)g-1][(2\lambda+3)g-2][(2\lambda+3)g-3]} \\
\times \Big(1-\frac{9 x^2}{g (2 \lambda +3)^2} \Big) \Big(\frac{4-9 g x^2}{g^2 (2 \lambda +3)^2}-\frac{4}{2 g \lambda +3 g}+1 \Big)\, .
\end{multline}
In the large $g$ limit the terms above are of order $1/16+o(1)$ uniformly in $n$, because the $x$-dependent terms appears in the combination $x^2/g$, which is at most of order $O(g^{-\frac{1}{2}-2\delta})$.

For the remaining terms of the binomial product, starting from \eqref{eqn-num-binomialAsy} again in this region, we find the prefactor $R$ of \eqref{eqn-num-binomialAsy} uniformly asymptotes to $\frac{4}{g \sqrt{2 \lambda +3}}$:
\begin{equation}
\begin{aligned}
   \frac{4}{g \sqrt{2 \lambda +3}} \frac{1}{\sqrt{1-\frac{x^2}{g}}} \frac{1}{\sqrt{1-\frac{9x^2}{(3+2\lambda)^2 g}}} & = \frac{4}{g \sqrt{2 \lambda +3}} \left( 1 + O\Big(\frac{x^2}{g} \Big) \right) \Big(1 +o(1) \Big) \\ 
   & =   \frac{4}{g \sqrt{2 \lambda +3}} \Big( 1 + o(1)  \Big)\, , 
\end{aligned}
\end{equation}
and the exponential term also asymptotes to a quadratic function in $x$,
\begin{equation}
    \begin{aligned}
        & g H\Big(\frac{1}{2} + \frac{x}{2\sqrt{g}}\Big) + (3g+2n) H\Big(\frac{1}{2} + \frac{1}{2} \frac{3 x}{(3+2\lambda)\sqrt{g}} \Big) \\
        = &\ g(\log 2 - \frac{x^2}{2g} ) + O\left(\frac{x^4}{g} \right) +(3g+2n) \left(\log 2 - \frac{1}{2}\Big( \frac{3 x}{(3+2\lambda)\sqrt{g}} \Big)^2 \right) + O\left(\frac{x^4}{(3+2\lambda)^2 g} \right)\\
        =&\ g (2+\lambda)  \log 4-\frac{(6+\lambda) x^2}{3 +2 \lambda} + O\left(g^{-4\delta} \right) + O\left(\frac{g^{-4\delta}}{(3+2\lambda)^2}\right)\, .
    \end{aligned}   
\end{equation}
The Taylor expansions above are legitimate because comparing to $\frac{1}{2}$, both $\frac{x}{2\sqrt{g}} = O(g^{-\frac{1}{2}-2\delta})$ and $\frac{3 x}{(3+2\lambda)\sqrt{g}} = \frac{1}{3+2\lambda} O(g^{-\frac{1}{2}-2\delta})$ are small, no matter what value $\lambda$ takes. As the result of the quadratic expansion equals to the uniform upper bound in \eqref{eqn-Num-expUpper}, we conclude the difference is negligible.

Therefore, we have proven, for $x \in [-g^{\frac{1}{4}-\delta}, \ g^{\frac{1}{4}-\delta}]$,
\begin{equation} \label{eqn-Num-numDominant}
             \binom{g}{g_1} \binom{3g-4+2n}{3g_1-2+n} \sim 2^{(2\lambda +4)g-4} \frac{1}{\sqrt{(2\lambda + 3)}  \frac{\pi g}{2}} e^{-(1+\frac{9}{2\lambda +3}) \frac{x^2}{2}}\, .
\end{equation}

To summarize, the dominant contribution of the sum \eqref{eqn:numeratorSum} comes from the region $x \in [-g^{\frac{1}{4}-\delta}, \ g^{\frac{1}{4}-\delta}]$, with contribution given in \eqref{eqn-Num-numDominant}. The reason is that when $x$ is in the complement $ [-(1-\varepsilon)\sqrt{g},(1-\varepsilon)\sqrt{g}]\backslash[-g^{\frac{1}{4}-\delta}, \ g^{\frac{1}{4}-\delta}] $, the upper bound \eqref{eqn-Num-UpperBound} implies the contribution to the sum \eqref{eqn:numeratorSum} is at most of order $\exp(-g^{\frac{1}{2}-2\delta}) (\sqrt{g})^2$, with the first factor of $\sqrt{g}$ being the length of the interval, and the second factor of $\sqrt{g}$ being the proportionality constant between $g_1$ and $x$, $\sum_{g_1} = \sqrt{g} \sum_{x}$. Compared with the dominant contribution \eqref{eqn-Num-numDominant}, we see that the summation in the complement region is exponentially suppressed.

This dominant contribution to the sum \eqref{eqn:numeratorSum} in the region $x \in [-g^{\frac{1}{4}-\delta}, \ g^{\frac{1}{4}-\delta}]$ can now be approximated as an Gaussian integral over $x$, with Jacobian factor $\frac{\sqrt{g}}{2}$ from changing sum over $g_1$ into integral over $dx$,
    \begin{equation}
    \label{eq:asymp_num}
    \begin{aligned}
        \sum\limits_{g_1=0}^{g} \binom{g}{g_1} \binom{3g-4+2n}{3g_1-2+n} 
        & \sim 2^{(2\lambda +4)g-4} \frac{1}{\sqrt{(2\lambda + 3)}  \frac{\pi g}{2}} \frac{\sqrt{g}}{2} \int^{+\infty}_{-\infty }  e^{-(1+\frac{9}{2\lambda +3}) \frac{x^2}{2}} \, dx  \\
        & = \frac{1}{\sqrt{\pi g (\lambda+6)}}2^{(2\lambda+4)g-4}\, .
    \end{aligned}
    \end{equation}
\end{proof}

\begin{prop} \label{prop:dnoAsy}
        The asymptotics of the denominator of \eqref{eq:ratio} is given by
        \begin{equation} \label{eqn:denomSum}
     \sum\limits_{k=0}^{3g-4} \frac{\binom{3g-4+2n}{n+k} \binom{6g-6}{2k+1} }{\binom{3g-4}{k}} = \sqrt{\frac{3}{3+\lambda}} 2^{(2\lambda+6)g -7} \big( 1+o(1) \big)\, ,
        \end{equation}
        with the error term uniformly small in $n$ as $g \rightarrow \infty$.
    \end{prop}
\begin{proof}
    Notice that all three binomial coefficients obtain their maximum at $k=\frac{3g-4}{2}$.
    We could naturally define
    \begin{equation*}
        \frac{k}{3g-4} = \frac{1}{2}\Big(1 + \frac{x}{\sqrt{3g-4}}\Big)\, ,
    \end{equation*} 
    and repeat the steps that lead to the previous proposition.

    As for the previous case, we divide the domain into the \textit{tail} region, where $x\in [-\sqrt{3g-4},-(1-\varepsilon) \sqrt{3g-4}] \cup [(1-\varepsilon) \sqrt{3g-4},\sqrt{3g-4}]$ and the rest, where $x \in [-(1-\varepsilon) \sqrt{3g-4},(1-\varepsilon) \sqrt{3g-4}]$. The constant $\varepsilon$ satisfies $0<\varepsilon <1$. Lemma \ref{lem:tailEstimate} and Remark \ref{rmk:tailEstimate} imply that the contribution of the tail region is negligible.

    To simplify the computation, we first express the second term in the numerator of \eqref{eqn:denomSum} as
    \begin{equation}
        \binom{6g-6}{2k+1} = \binom{6g-8}{2k} r_1(k;g)
    \end{equation}
    with
    \begin{equation}
        r_1(k;g) = \frac{6 (g-1) (6 g-7)}{(2 k+1) (6 g-2 k-7)} = \frac{6 (g-1) (6 g-7)}{9 (g-1)^2 - (3 g-4) x^2} \, .
    \end{equation}
    $r_1$ obtains its maximum value at the boundary $x^2 = (1-\varepsilon)\sqrt{3g-4}$:
    \begin{equation} \label{eqn-Den-r1Bound}
        r_1(k;g) \leq \frac{4}{\varepsilon (2-\varepsilon)} \left(1 + O\left(\frac{1}{g}\right) \right) \, .
    \end{equation}

    Using lemma \eqref{lem-Estimate}, we find
    \begin{equation}
    \begin{aligned}
       & \frac{\binom{3g-4+2n}{n+k} \binom{6g-8}{2k} }{\binom{3g-4}{k}} \sim R_2(k;g,\lambda) \\ 
       & \times \exp\Big[ (3g-4+2n) H\Big(\frac{1}{2}+\frac{\sqrt{3 g-4} x}{2g (2 \lambda +3)-8} \Big) + (3g-4) H\Big(\frac{1}{2} + \frac{x}{2\sqrt{3 g-4}}\Big) \Big] \, , 
    \end{aligned}
    \end{equation}
    where
    \begin{equation}
        R_2(k;g,\lambda) = \frac{1}{\sqrt{\pi}} \frac{1}{\sqrt{(3g-4) \left(1-\frac{x^2}{2 g \lambda +(3g-4)}\right)+2g \lambda }}\, .
    \end{equation}
    We only have two $H$ functions instead of three in the exponential, because the $\binom{6g-8}{2k}$ term in the numerator and the $\binom{3g-4}{k}$ term in the denominator can be expressed by $H$ with the same argument.
    
We find, for $x \in [-(1-\varepsilon)\sqrt{3g-4},(1-\varepsilon)\sqrt{3g-4}]$, the remainder term $R_2(k;g,\lambda)$ is bounded by the value at $x=(3g-4)(1-\varepsilon)^2$,
\begin{equation} \label{eqn-Den-R2Bound}
    R_2(k;g,\lambda) \leq \frac{1}{\sqrt{\pi}} \frac{1}{\sqrt{(3g-4) \left(1-\frac{(3g-4)(1-\varepsilon)^2}{2 g \lambda +(3g-4)}\right)+2g \lambda }} \leq \frac{1}{\sqrt{\pi(3g-4)}} \frac{1}{\sqrt{1-(1-\varepsilon)^2}}\, ,
\end{equation}
where in the last step we substitute the maximal value at $\lambda=0$ to obtain a uniform upper bound.

The exponential terms are bounded using \eqref{eqn-HinEquality},
\begin{equation} \label{eqn-Den-ExpHEstimate}
    \begin{aligned}
      & \exp\Big[ (3g-4+2n) H\Big(\frac{1}{2}+\frac{\sqrt{3 g-4} x}{2g (2 \lambda +3)-8} \Big) + (3g-4) H\Big(\frac{1}{2} + \frac{x}{2\sqrt{3 g-4}} \Big) \Big] \\
      \leq & 4^{(3+\lambda)g-4} \exp \Big( -\frac{x^2 (g (\lambda +3)-4)}{ (2 \lambda +3)g-4} \Big) \leq  4^{(3+\lambda)g-4} \exp \Big(- \frac{x^2}{2} \Big) \, . \\
\end{aligned}
\end{equation}
where we bound the coefficient of exponential as $1/2$. As one can verify that
\begin{equation}
    \frac{(g (\lambda +3)-4)}{ (2 \lambda +3)g-4} - \frac{1}{2} = \frac{3 g-4}{g (4 \lambda +6)-8} >0
\end{equation}
when $g>4/3$.

The results above show that the sum is uniformly bounded by
\begin{equation} \label{eqn-Den-UpperBound}
    \sum\limits_{k=0}^{3g-4} \frac{\binom{3g-4+2n}{n+k} \binom{6g-6}{2k+1} }{\binom{3g-4}{k}}  \leq \frac{\text{const}}{\sqrt{3g-4}} \cdot 4^{(3+\lambda)g} \exp \Big( - \frac{x^2}{2} \Big) \, .
\end{equation}
where the constant only depends on $\varepsilon$. The precise value of the constant can be found using \eqref{eqn-Den-r1Bound}, \eqref{eqn-Den-R2Bound}, and \eqref{eqn-Den-ExpHEstimate}, but it is not needed for our purposes.

Now we estimate the contributions to the sum in the region $x \in [-(3g-4)^{\frac{1}{4}-\delta}, \ (3g-4)^{\frac{1}{4}-\delta}]$, with $0<\delta<1/4$. In this region, a finer estimation can be made, since $x^2 = O(g^{\frac{1}{2}-2\delta})$. We find in this regime the term in the denominator of \eqref{eqn:denomSum} can be approximated as
  \begin{equation} \label{eqn:denomAsy1}
     \binom{3g-4}{k}= 2^{3g-4} \cdot e^{-\frac{x^2}{2}} \cdot \sqrt{\frac{2}{\pi(3g-4)}} \cdot \Big(1+ O(g^{-4\delta}) \Big)\sim 2^{3g-4} \cdot e^{-\frac{x^2}{2}} \cdot \sqrt{\frac{2}{3\pi g}} \, .
    \end{equation}
    The first term in the in the numerator of \eqref{eqn:denomSum} reads
    \begin{equation} \label{eqn:denomAsy3}
       \binom{3g-4+2n}{k+n} \sim 2^{(3+2\lambda) g -\frac{7}{2}} \sqrt{\frac{1}{\pi g(3+2\lambda)}} e^{- \frac{3}{3+2\lambda} \frac{x^2}{2}} \, ,
    \end{equation}
    while for the second term $\binom{6g-6}{2k+1} = r_1 \binom{6g-8}{2k}$ in the numerator of \eqref{eqn:denomSum}, we find
    \begin{equation} \label{eqn:denomAsy2}
        \binom{6g-8}{2k} \sim 2^{6g-8} \frac{1}{\sqrt{3\pi g}} e^{- x^2}\, , \quad  r_1 = 4 \Big(1 + O(\frac{x^2}{g} )\Big) =  4 \Big(1 + O(g^{-\frac{1}{2}-2\delta}) \Big)\, .
    \end{equation}
    Therefore, using \eqref{eqn:denomAsy1}, \eqref{eqn:denomAsy3} and \eqref{eqn:denomAsy2}, we conclude that when $x \in [-(3g-4)^{\frac{1}{4}-\delta}, \ (3g-4)^{\frac{1}{4}-\delta}]$,
    \begin{equation}
        \dfrac{\binom{3g-4+2n}{n+k} \binom{6g-6}{2k+1} }{\binom{3g-4}{k}} 
       \sim 2^{(2\lambda+6)g -6  } \frac{1}{ \sqrt{\pi (3+2\lambda)g }} e^{-\frac{3+\lambda}{3+2\lambda} x^2}\, .
    \end{equation}
    This estimate, combined with the uniform upper bound given in equation \eqref{eqn-Den-UpperBound}, suggests that the dominant contribution to the sum comes from the central region $x \in [-(3g-4)^{\frac{1}{4}-\delta}, \ (3g-4)^{\frac{1}{4}-\delta}]$, as similarly demonstrated in the proof of Proposition \ref{prop:numAsy}. Contributions outside the central region are exponentially suppressed compared to those within the region.

    We can now approximate the summation in the region $x \in [-(3g-4)^{\frac{1}{4}-\delta}, \ (3g-4)^{\frac{1}{4}-\delta}]$ as a Gaussian integral over $x$, with Jacobian $\frac{\sqrt{3g-4}}{2}$.
    \begin{equation}
    \label{eq:asymp_denom}
    \begin{aligned}
         \sum\limits_{k=0}^{3g-4} \dfrac{\binom{3g-4+2n}{n+k} \binom{6g-6}{2k+1} }{\binom{3g-4}{k}} 
        & \sim 2^{(2\lambda+6)g -6  } \frac{1}{ \sqrt{\pi (3+2\lambda)g }} \frac{\sqrt{3g-4}}{2} \int_{-\infty}^{+\infty} e^{-\frac{3+\lambda}{3+2\lambda} x^2} dx\\
        & = \sqrt{\frac{3g-4}{(3+\lambda)g}} 2^{(2\lambda+6)g -7} \sim \sqrt{\frac{3}{3+\lambda}} 2^{(2\lambda+6)g -7} \, .
    \end{aligned}
    \end{equation}
\end{proof}
Eventually, we have the following effortless proof for the main theorem:
\begin{proof}[Proof of Theorem~ \ref{thm:conjMain}]Direct calculation by Proposition \ref{prop:numAsy} and \ref{prop:dnoAsy}.
\end{proof}

\nocite{*}
\bibliography{ref}

\end{document}